\documentclass[a4paper,10pt,reqno, english]{amsart}  
\usepackage[utf8]{inputenc}
\usepackage[T1]{fontenc}
\usepackage{amsmath,amsthm}
\usepackage{amsfonts,amssymb,enumerate}
\usepackage{url,paralist}
\usepackage{mathtools}  
\usepackage[colorlinks=true,urlcolor=blue,linkcolor=red,citecolor=magenta]{hyperref}
\usepackage{enumerate}
\usepackage{anysize}

\theoremstyle{plain}
\newtheorem{theorem}{Theorem}[section]
\newtheorem{lemma}[theorem]{Lemma}

\newtheorem{corollary}[theorem]{Corollary}

\newtheorem{conjecture}[theorem]{Conjecture}
\newtheorem*{theorem*}{Theorem}

\newtheorem*{claim*}{Claim}

\theoremstyle{definition}

\newtheorem{remark}[theorem]{Remark}

\newcommand{\R}{\mathbb{R}}
\newcommand{\Z}{\mathbb{Z}}
\newcommand{\F}{\mathbb{F}}
\newcommand{\E}{\mathrm{E}}  
\newcommand{\B}{\mathrm{B}}
\newcommand{\dist}{\mathrm{dist}}
\newcommand\Sym{\mathfrak S}

\begin{document}

\title[Barycenters and Topological Tverberg via Constraints]
{Barycenters of Polytope Skeleta and Counterexamples to the Topological Tverberg Conjecture, via Constraints}
 
\author[Blagojevi\'c]{Pavle V. M. Blagojevi\'{c}} 
\thanks{P.B.\ was supported by the DFG
        Collaborative Research Center TRR~109 ``Discretization in Geometry and Dynamics''
        and by grant ON 174024, Serbian Ministry of Education and Science.}
\address[PVMB]{Mat. Institut SANU, Knez Mihailova 36, 11001 Beograd, Serbia\hfill\break
\mbox{\hspace{4mm}}Inst. Math., FU Berlin, Arnimallee 2, 14195 Berlin, Germany}
\email{blagojevic@math.fu-berlin.de} 
\author[Frick]{Florian Frick}
\thanks{F.F.\ was supported by DFG via the Berlin Mathematical School.}
\address[FF]{Dept.\ Math. Sciences, Carnegie Mellon University, Pittsburgh, PA 15213, USA}
\email{frick@cmu.edu} 
\author[Ziegler]{G\"unter M. Ziegler} 
\thanks{G.M.Z.\ received funding from the ERC project  
no. 247029-\emph{SDModels} and by DFG via the Research Training Group Methods for Discrete Structures}  
\address[GMZ]{Inst. Math., FU Berlin, Arnimallee 2, 14195 Berlin, Germany} 
\email{ziegler@math.fu-berlin.de}

\date{October 27, 2015; revised April 4, 2016 and February 14, 2017}

\dedicatory{Dedicated to the memory of Ji\v{r}\'{\i} Matou\v{s}ek}

\maketitle


\begin{abstract}
Using the authors' 2014 ``constraint method,''
we give a short proof for a 2015 result of Dobbins on representations of a point in a polytope as
the barycenter of points in a skeleton, and
show that the ``$r$-fold Whitney trick'' of Mabillard and Wagner (2014/2015)
implies that the Topological Tverberg Conjecture for $r$-fold intersections
fails dramatically for all $r$ that are not prime powers.
\end{abstract}
 
\section{Introduction}
\label{sec : Introduction}

The \emph{Topological Tverberg Theorem} states that for any $d\ge1$, prime power $r\ge2$, and $N:=(r-1)(d+1)$
every continuous map of the simplex $\Delta_N$ to $\R^d$ identifies points from $r$ vertex-disjoint faces of $\Delta_N$.
This result was established for prime $r$ by B{\'a}r{\'a}ny, Shlosman \& Sz{\H{u}}cs \cite{Barany1981} in 1981,
and for prime powers $r$ in famous unpublished work by \"Ozaydin \cite{Oezaydin1987} from 1987. 
The belief that the result should be equally valid for values $r\ge6$ that are not prime powers,
known as the \emph{Topological Tverberg Conjecture},
was left open, as “a holy grail of topological combinatorics (Kalai \cite{Kalai2015-blog}).

The ``constraint method,'' introduced by the authors in 2014 \cite{Blagojevic2014}, shows that the 1981/1987
Topological Tverberg Theorem implies virtually all subsequent extensions and sharpenings
that were previously viewed as substantial independent results, such as the ``Colored Tverberg Theorem'' of \v{Z}ivaljevi\'c \& Vre\'cica \cite{Zivaljevic1992} or the ``Generalized Van Kampen--Flores Theorem'' of Sarkaria \cite{Sarkaria1991-1} and Volovikov \cite{Volovikov1996-2}. 
In this latter case, this implication had already been observed by Gromov \cite{Gromov2010} in 2010.

In this paper, we first demonstrate the mechanism and the power of the constraint method outside the classical ``Tverberg type theorems,'' by a simple proof of a 2015 \emph{Inventiones} result of Dobbins \cite{Dobbins2014}.
(This was the content of our research announcement \cite{Blagojevic2014-2}.)

More importantly, we get a new quality by showing that the constraint method not only yields substantial consequences
by ``using the Topological Tverberg Theorem'': It also allows us to derive from recent deep work of Mabillard and Wagner the failure of the Topological Tverberg Conjecture for all $r\ge6$ that are not prime powers.

Mabillard and Wagner announced their work on the ``$r$-fold Whitney trick'' at the \emph{SoCG} conference in June 2014 \cite{Mabillard2014} with the explicit intention of deriving counterexamples to the Topological Tverberg Conjecture by combining it with a result of \"Ozaydin~\cite{Oezaydin1987}.
At that time, it appeared that due to the codimension condition implicit in both the classical and the $r$-fold Whitney trick, their work could not yield the desired counterexamples to the Topological Tverberg Conjecture.
This was reiterated in lectures by Wagner in Copenhagen (November 2014) and by Mabillard in Berlin (January 2015).
Thus the announcement of counterexamples by the second author in February 2015 in Oberwolfach~\cite{Frick2015}
came as a surprise (as documented e.g.\ on the Kalai blog~\cite{Kalai2015-blog}). 
Now, after the release of the full journal version of the Mabillard--Wagner work on ArXiv in August 2015~\cite{Mabillard2015}, we here present our short and simple proof that their work implies the failure of the Topological Tverberg Conjecture, based on the research announcement~\cite{Frick2015}, in Section~\ref{sec:TTCfails}.
Finally we comment on the degree of failure of the Topological Tverberg Conjecture, and pose a new conjecture, in Section~\ref{sec:failure}.

\section{Constraints, an example}

The main idea and the power of the so called ``constraint method'' becomes apparent in the following brief proof of a result by Dobbins from 2015.

\begin{theorem}[Dobbins {\cite[Thm.\,1]{Dobbins2014}}]
\label{thm:dobbins}
	Let $P$ be a $d$-dimensional polytope, $p \in P$, and let 
	$k\ge1$ and $r\ge1$ be integers with $kr \ge d$. 
	Then there are points $p_1,\dots , p_r$ in the $k$-skeleton 
	$P^{(k)}$ of~$P$ with barycenter \[p = \tfrac{1}{r}(p_1 + \dots +p_r).\]
\end{theorem}

\begin{proof} 
	We can assume that $p = 0$, and that $p$ is in the interior of~$P$: 
	Otherwise we could restrict to a proper face of~$P$ with the 
	origin in its relative interior. 
	Let first $r$ be prime. 
	Consider the linear space $W_r = \{(x_1, \dots, x_r) \in \R^r : \sum_{i=1}^r x_i = 0\}$ and its $d$-fold 
	direct sum $W_r^{\oplus d} \subseteq \R^{d\times r}$ that is a subspace of codimension~$d$ inside $\R^{d\times r}$. 
	Then $C = P^r \cap W_r^{\oplus d}$ is a polytope of dimension~$(r-1)d$. 
	The $(r-1)$-skeleton $C^{(r-1)}$ of~$C$ is homotopy equivalent to a wedge of $(r-1)$-spheres and thus is $(r-2)$-connected.
	
	Let $(x_1, \dots, x_r) \in C^{(r-1)}$, then at least one $x_i$ lies in~$P^{(k)}$: 
	Suppose for contradiction that $x_i \notin P^{(k)}$ for all $i = 1, \dots, r$. 
	For each $x_i$ let $\sigma_i$ be the inclusion-minimal face of $P$ with $x_i \in \sigma_i$. 
	Consequently, $\dim \sigma_i \ge k+1$ for all $i = 1, \dots, r$.
	Each face of~$C$ is of the form $(\tau_1 \times \dots \times \tau_r) \cap W_r^{\oplus d}$ with the $\tau_i$ faces of~$P$. 
	The point $(x_1, \dots, x_r)$ lies in the face $(\sigma_1 \times \dots \times \sigma_r) \cap W_r^{\oplus d}$ but in no proper subface. 
	Now 
	\[
	\dim \left((\sigma_1 \times \dots \times \sigma_r) \cap W_r^{\oplus d}\right) \ge r(k+1)-d\ge r.
	\]
	Thus, $(\sigma_1 \times \dots \times \sigma_r) \cap W_r^{\oplus d}\notin C^{(r-1)}$, 
	which is a contradiction.

	Consider the constraint
	 function $\psi\colon  C^{(r-1)} \longrightarrow W_r$ defined by
	\[
	 (x_1, \dots, x_r) \longmapsto (\dist(x_1, P^{(k)})-\tfrac1rD, \dots, \dist(x_r, P^{(k)})-\tfrac1rD),
	\]
	where $D=\sum_{i=1}^r \dist(x_i, P^{(k)})$.
	If the action of the symmetric group $\Sym_r$ is given on $P^r$ and $\R^r$ by the permutation of factors in the product then the subspaces $C^{(r-1)}$ and $W_r$ are invariant and inherit the action.
	The continuous function $\psi$ is $\Sym_r$-equivariant.
    The vector space~$W_r$ is $(r-1)$-dimensional and $C^{(r-1)}$ is $(r-2)$-connected, 
	so $\psi$ has a zero by a theorem of Dold \cite[Thm.\,6.2.6]{Matousek2008} applied to the subgroup~$\Z/r$ of~$\Sym_r$, which acts freely on $W_r {\setminus} \{0\}$. 
	Let $(p_1, \dots, p_r) \in C^{(r-1)}$ with $\psi(p_1, \dots, p_r) = 0$. 
	This is equivalent to $\dist(p_1, P^{(k)}) = \dots = \dist(p_r, P^{(k)})$. 
	Since, there is an index~$i$ such that $p_i \in P^{(k)}$ it follows that $\dist(p_i,P^{(k)}) = 0$. 
	Consequently, all $p_j$ satisfy $\dist(p_j,P^{(k)}) = 0$ and hence are in $P^{(k)}$.
	Since $C^{(r-1)}\subseteq W_r^{\oplus d}$ we have $p_1 + \dots + p_r = 0$.
	
	The case for general $r$ follows by a simple induction with respect to the number of prime divisors, as in~\cite{Dobbins2014}: 
	Suppose $r = q_1 \cdots q_t$ with $q_i$ prime and the theorem holds for any number $r$ that is a product of at most $t-1$ primes. 
	Let $m = q_2 \cdots q_t$.
	Since $m\cdot q_1k = rk \ge d$, there are $m$ points $x_1, \dots,x_m$ in $P^{(q_1k)}$ with $p = \frac{1}{m}(x_1+\dots+x_m)$. 
	Each $x_i$ is contained in a $(q_1k)$-face $\sigma_i$ of~$P$. 
	Thus, there are $q_1$ points $y_i^{(1)}, \dots, y_i^{(q_1)}$ in the $k$-skeleton $\sigma_i^{(k)}$ of~$\sigma_i$ with $x_i = \frac{1}{q_1}(y_i^{(1)}+\dots+y_i^{(q_1)})$.
	In particular, the $y_i^{(j)}$ are also contained in~$P^{(k)}$ and $p = \frac{1}{m} \sum_{i=1}^m \frac{1}{q_1} \sum_{j=1}^{q_1} y_i^{(j)} = \frac{1}{r} \sum_{i=1}^m \sum_{j=1}^{q_1} y_i^{(j)}$.	
	\end{proof}
	
\section{Three Main Ingredients}\label{sec:GeneralizedVKF}

Imre B\'ar\'any, in a 1976 letter to Helge Tverberg, asked whether Tverberg's classical 1966 theorem~\cite{Tverberg1966} 
would generalize from affine maps to continuous maps, that is, whether 
for any continuous map $f\colon\Delta_N\longrightarrow~\R^d$
with $N=(r-1)(d+1)$ there are
$r$ points from disjoint faces of the $N$-dimensional simplex that map to the same point in~$\R^d$.
The problem was posed by Tverberg at an Oberwolfach conference in May 1978, and first appeared in print as \cite[Problem~84]{Gruber1979} in~1979.
For $r=2$ this had already been established by Bajm\'oczy and B\'ar\'any \cite{Bajmoczy1979}.
In 1981 B\'ar\'any, Shlosman and Sz\H ucs answered B\'ar\'any's question positively for the case of primes, and later in 1987 \"Ozaydin for the case of prime powers.
This result is known as the \emph{Topological Tverberg Theorem}.

\begin{theorem}[B\'ar\'any, Shlosman and Sz\H ucs \cite{Barany1981}, \"Ozaydin \cite{Oezaydin1987}]
	\label{thm:topological_Tverberg}
	Let $d\geq 1$ be an integer, let $r$ be a power of prime, and $N=(r-1)(d+1)$.
	 Then for any continuous map $f\colon \Delta_N \longrightarrow \R^d$ there are $r$ pairwise disjoint faces $\sigma_1, \dots, \sigma_r$ of the simplex $\Delta_N$ with $f(\sigma_1) \cap \dots \cap f(\sigma_r) \ne\emptyset$.
\end{theorem}

This result for prime $r$ follows from the connectivity of the deleted product configuration spaces $(\Delta_N)^{\times r}_{\Delta(2)}$ 
and freeness of the $\Z/r$-action on the sphere $S(W_r^{\oplus d})$.
The prime power case $r=p^k$ is due to the fixed-point free action of the elementary abelian group $G=(\Z/p)^k$ on the sphere $S(W_r^{\oplus d})$
together with a corollary to the Localization Theorem for elementary abelian groups (cf.\ \cite[Thm.\,III.3.8]{tomDieck:TransformationGroups}):
For any finite-dimensional $G$-CW complex $X$ the following are equivalent,  
\begin{enumerate}[(i)]
	\item The fixed-point set $X^G$ is non-empty,
	\item The map  $H^*(\B G;\F_p)\longrightarrow H^*(\E G\times_GX;\F_p)$,
	  induced by the projection $\E G\times_GX\longrightarrow\B G$, is injective. 
\end{enumerate}

The \emph{Topological Tverberg Conjecture} asserts an affirmative answer to B\'ar\'any's question, that is that 
Theorem~\ref{thm:topological_Tverberg} remains true outside of the prime power case.
It was widely believed to be true and that $r$ having only one prime divisor was simply an artifact of the proof method; 
see for example Matou\v sek~\cite[p.~162]{Matousek2008}: ``It seems likely that this theorem remains true for all~$p$.'' 
It was the insight and foresight of Mabillard and Wagner to work 
against the grain of the field and develop their ``$r$-fold Whitney trick.''
They envisioned that this could be combined with the work of Özaydin~\cite{Oezaydin1987} to construct counterexamples to the
Topological Tverberg Conjecture. In fact, only one more ingredient is missing that when all combined yield the desired counterexamples.
Here we collect these three main ingredients:
\smallskip

\noindent
\textbf{Ingredient 1: Topological Tverberg implies Generalized Van Kampen--Flores.} 
The following Generalized Van Kampen--Flores Theorem turns out to be an easy corollary of the Topological Tverberg Theorem, 
if we use the constraint method \cite[Thm.\,6.3]{Blagojevic2014};
this was apparently first pointed out by Gromov \cite[Sect.~2.9c]{Gromov2010}, 
whose sketch can be seen as a first instance of the constraint method ``at work.''
The Generalized Van Kampen--Flores Theorem was originally obtained,
with significantly more involved proofs, by Sarkaria for primes and by Volovikov for prime powers.

\begin{theorem}[Sarkaria \cite{Sarkaria1991-1}, Volovikov \cite{Volovikov1996-2}]
	\label{thm:genrelized_van_Kampen_Flores}
	Let $d\ge1$ be an integer, let $r$ be a power of a prime, set $N =(r-1)(d+2)$, and let $k\ge\lceil \tfrac{r-1}rd\rceil$. 
	 Then for any continuous map $f\colon \Delta_N \longrightarrow \R^d$ there are $r$ pairwise disjoint faces $\sigma_1, \dots, \sigma_r$ of the $k$-th skeleton $\Delta_N^{(k)}$ with $f(\sigma_1) \cap \dots \cap f(\sigma_r) \ne\emptyset$.
\end{theorem}
\begin{proof}
	Let $g\colon \Delta_N\longrightarrow\R^{d+1}$ be a continuous function defined by $g(x)=(f(x),\dist(x, \Delta_N^{(k)}))$.
	The Topological Tverberg Theorem applied to the function $g$ yields a collection of points $x_1,\dots,x_r$ in pairwise disjoint faces $\sigma_1,\dots,\sigma_r$ with $f(x_1) = \dots = f(x_r)$ and $\dist(x_1,\Delta_N^{(k)})=\dots=\dist(x_r,\Delta_N^{(k)})$.
    We can assume that the all $\sigma_i$'s are inclusion-minimal with the property that $x_i\in\sigma_i$, that is, $\sigma_i$ is the unique face with $x_i$ in its relative interior.
    Now, if one of the $\sigma_i$'s were in $\Delta_N^{(k)}$, then we would have
$\dist(x_1,\Delta_N^{(k)})=0$, implying that $\dist(x_1,\Delta_N^{(k)})=\dots=\dist(x_r,\Delta_N^{(k)})=0$, and consequently the theorem would be proved.
    
    Let us assume the contrary, that no $\sigma_i$ is in $\Delta_N^{(k)}$, i.e., $\dim\sigma_1\geq k+1,\dots,\dim\sigma_r\geq k+1$. 
    Since the faces $\sigma_1,\dots,\sigma_r$ are pairwise disjoint we have that
    \[
    N+1=|\Delta_N|\geq |\sigma_1|+\dots+|\sigma_r|\geq r(k+2)\geq r\big(\lceil \tfrac{r-1}rd\rceil+2\big)\geq (r-1)(d+2)+2=N+2.
    \]
    This is a contradiction. 
\end{proof}

The essential conclusion is that: 
\emph{If the Topological Tverberg Theorem holds for an integer~$r$ and dimension~${d+1}$,
 then the Generalized Van Kampen--Flores Theorem also holds for the same integer~$r$ and dimension~$d$}.
\smallskip

\noindent 
\textbf{Ingredient 2: Equivariant maps yield maps without points of \emph{r}-fold coincidence.} 
This is a highly nontrivial recent
result of Mabillard and Wagner. In order to state it we need 
the \emph{pairwise deleted $r$-fold product} of simplicial complex~$K$, defined as
\[
K^{\times r}_{\Delta(2)} = \{(x_1,\dots,x_r) \in \sigma_1 \times \dots \times \sigma_r \: | \: \sigma_i \text{ face of } K, \sigma_i \cap \sigma_j = \emptyset \text{ for } i \neq j\}.
\]
The space $K^{\times r}_{\Delta(2)}$ is a polytopal $\Sym_r$-complex: Its faces are products of simplices; the symmetric group acts on it by permuting factors.
Recall that $W_r$ was defined above as $\{(x_1, \dots, x_r) \in \R^r : \sum_{i=1}^r x_i = 0\}$,
with the $\Sym_r$-action permuting the coordinates.

\begin{theorem}[Mabillard \& Wagner {\cite[Thm.\,3]{Mabillard2014} \cite[Thm.\,7]{Mabillard2015}}]
	\label{thm:mabillard-wagner}
	Let $r\ge 2$ and $k \ge 3$ be integers, and let $K$ be a simplicial complex of dimension~$(r-1)k$. 
	Then the following statements are equivalent:
	\begin{compactenum}[\rm(i)]
		\item There exists an $\Sym_r$-equivariant map $K^{\times r}_{\Delta(2)} \longrightarrow S(W_r^{\oplus rk})$.
		\item There exists a continuous map $f \colon K \longrightarrow \R^{rk}$ such that for any $r$ pairwise disjoint faces $\sigma_1, \dots, \sigma_r$ of~$K$ we have that
		$f(\sigma_1) \cap \dots \cap f(\sigma_r) = \emptyset$.
	\end{compactenum}
\end{theorem} 

\noindent
\textbf{Ingredient 3: Constructing equivariant maps outside of the prime power case.} It is a simple corollary of the
work of \"Ozaydin~\cite{Oezaydin1987} that an $\Sym_r$-equivariant map $K^{\times r}_{\Delta(2)} \longrightarrow S(W_r^{\oplus d})$ exists
if the dimension gap $d-\dim K$ is sufficiently large and $r$ has at least two distinct prime divisors:
  
\begin{theorem}
	Let $r \ge 6$ be an integer that is not a prime power, and let $k \ge 3$ be an integer. 
	Let $K$ be a simplicial complex of dimension at most $(r-1)k$. Then 
	there exists an $\Sym_r$-equivariant map $K^{\times r}_{\Delta(2)} \longrightarrow S(W_r^{\oplus rk})$.
\end{theorem}

\begin{proof}
	Let $\mathrm{E}_{M}\Sym_r$ denote an $M$-dimensional free $\Sym_r$-complex that is $(M-1)$-connected.
	For example, $\mathrm{E}_{M}\Sym_r$ can be modeled by the $(M+1)$-fold join $(\Sym_r)^{*M+1}$ where, with usual abuse of notation,  
	$\Sym_r$ is viewed as
a zero-dimensional simplicial complex whose vertices are elements of the group~$\Sym_r$.
	The group action is given by multiplication from the left.
	(This is an instance of the Milnor construction for classifying spaces.)
	
	The free $\Sym_r$-space $K^{\times r}_{\Delta(2)}$ has dimension at most $d=r(r-1)k$.
	Consequently, there exists an $\Sym_r$-equivariant map $K^{\times r}_{\Delta(2)}\longrightarrow \mathrm{E}_{r(r-1)k}\Sym_r$.
	Since $r$ is not a prime power, a result of \"Ozaydin~\cite[Thm.\,4.2]{Oezaydin1987} implies the existence of an 
	$\Sym_r$-equivariant map $\mathrm{E}_{r(r-1)k}\Sym_r\longrightarrow S(W_r^{\oplus rk})$.
	Composing these two maps we get an $\Sym_r$-equivariant map $K^{\times r}_{\Delta(2)}\longrightarrow  S(W_r^{\oplus rk})$.
\end{proof}
  
\section{Counterexamples to the Topological Tverberg Conjecture}\label{sec:TTCfails}

A straightforward combination of Ingredient 3 and 
Ingredient 2 -- the theorem of Mabillard and Wagner~{--} shows that 
the Generalized Van Kampen--Flores Theorem does not hold in general:

\begin{theorem}
	\label{thm:counter_van_kampen_flores}
	Let $r \ge 6$ be an integer that is not a prime power, and let $k \ge 3$ be an integer. 
	Then for any simplicial complex $K$ of dimension at most $(r-1)k$ there exists a continuous map 
	$f\colon K \longrightarrow \R^{rk}$ such that for any $r$ pairwise disjoint faces 
	$\sigma_1, \dots, \sigma_r$ of~$K$ we have that $f(\sigma_1) \cap \dots \cap f(\sigma_r) = \emptyset$.
\end{theorem} 
 
It is worth stressing that the theorem above provides counterexamples to an extension of 
the Generalized Van Kampen--Flores Theorem for any $r$ that is not a prime power, and for the $(r-1)k$-skeleton of the simplex $\Delta_N$ for \emph{any}~$N$.
Since the 
Generalized Van Kampen--Flores Theorem is a higher-dimensional, multi-intersection analogue
of ``$K_5$ is non-planar,'' in the sense that the case $r=2$ and $d=2$ reduces to this statement,
the theorem above implies that \emph{some} higher-dimensional, multi-intersection analogue
of the non-planarity of $K_5$ exists if and only if $r$ is a power of a prime. 

We have already pointed out that, if the Topological Tverberg Conjecture holds for some~$r$, then the Generalized Van Kampen--Flores Theorem holds for the same~$r$. 
Since Theorem~\ref{thm:counter_van_kampen_flores} contradicts the $r$-fold Van Kampen--Flores Theorem for $r$ not a power of a prime, the Topological Tverberg Conjecture must also fail for those~$r$.

\begin{theorem}
	Let $r \ge 6$ be an integer that is not a prime power, and let $k \ge 3$ be an integer. 
	Let $N =(r-1)(rk+2)$. 
	Then there exists a continuous map $F \colon \Delta_N \longrightarrow \R^{rk+1}$ such that for any $r$ pairwise disjoint faces $\sigma_1, \dots, \sigma_r$ of~$\Delta_N$ we have that $F(\sigma_1) \cap \dots \cap F(\sigma_r) = \emptyset$.
\end{theorem}

\begin{proof}
	Take $K$ to be the $(r-1)k$-skeleton of~$\Delta_N$. Then Theorem~\ref{thm:counter_van_kampen_flores} guarantees the existence
	of a map $f \colon K \longrightarrow \R^{rk}$ without a point of $r$-fold coincidence among its pairwise disjoint faces. There is no
	obstruction to continuously extending the map $f$ to the entire simplex~$\Delta_N$. Thus $f$ shows that the Generalized Van Kampen--Flores
	Conjecture, that is, the statement that Theorem~\ref{thm:genrelized_van_Kampen_Flores} holds beyond the prime power case, 
	does not hold for $r$ and dimension~$rk$. By the essential conclusion of Ingredient 1 this yields a counterexample $F$ to the 
	Topological Tverberg Conjecture for intersection multiplicity~$r$ and in dimension~${rk+1}$. 
\end{proof}

If the Topological Tverberg Conjecture holds for $r$ pairwise disjoint faces and dimension~$d+1$, 
then it also holds for dimension~$d$ and the same number of faces. This is a simple fact that follows easily
from the constraint method, but was also pointed out earlier by de~Longueville~\cite[Prop.~2.5]{DeLongueville2002}. 
In fact we will show stronger dimension reduction results in the next section.
Thus, we are only interested in low-dimensional counterexamples. If 
$r$ is not a prime power then the Topological Tverberg Conjecture fails for 
dimensions~$3r+1$ and above. Hence, the smallest counterexample our
construction yields is a continuous map $\Delta_{100} \longrightarrow \R^{19}$ such 
that any six pairwise disjoint faces have images that do not intersect in a 
common point.

In subsequent work, Mabillard \& Wagner \cite{Mabillard2015} constructed counterexamples in dimension~$3r$ with a new method, called 
\emph{prismatic maps}.  
Avvakumov, Mabillard, Skopenkov \& Wagner \cite{MabillardWagner-III} further improved this, with other methods, to
get counterexamples in dimension $2r$; in particular, their smallest 
counterexample is a map $\Delta_{65} \longrightarrow \R^{12}$. 

\begin{remark}
	Theorem~\ref{thm:counter_van_kampen_flores} immediately implies that the colored Tverberg theorem of ``type B'' 
	by Vre\'cica~\& \v{Z}ivaljevi\'c~\cite{Vrecica1994} fails for any number of faces $r$ 
	that is not a prime power. 
	They proved that any continuous map $f\colon [2r-1]^{*s} \longrightarrow \R^d$ has an $r$-fold
	Tverberg point for $s \ge \frac{r-1}{r}d + 1$ and $r$ a prime power. This result is colored in the sense
	that the discrete sets $[2r-1]$ are thought of as color classes and thus intersecting faces do not have
	monochromatic edges. (The case $r=2$ and $d=2$ implies the non-planarity of~$K_{3,3}$.) 
	A result is of ``type B'' in the language of Vre\'cica~\& \v{Z}ivaljevi\'c if the complex
	has dimension less than~$d$. It is this codimension requirement that guarantees that 
	Theorem~\ref{thm:counter_van_kampen_flores} may be applied:
	For $d = rk$ with $k \ge 3$ and $s = (r-1)k+1$ the complex $[2r-1]^{*s}$ has dimension~$s-1=(r-1)k\le d-3$. 
	
	The constraint method of~\cite{Blagojevic2014} yields a combinatorial reduction of this result to the 
	Topological Tverberg Theorem like in the case of the Generalized Van Kampen--Flores result. 
	Thus the refutation of this
	result also implies that the Topological Tverberg Conjecture fails outside the prime power case.
\end{remark}

\section{How badly does the Topological Tverberg Conjecture fail?}\label{sec:failure}

Let $N_r(d)$ be the minimal integer~$N$ such that for any continuous map $f\colon \Delta_N \longrightarrow \R^d$ there are $r$ pairwise disjoint faces $\sigma_1, \dots, \sigma_r$ with the property that $f(\sigma_1) \cap \dots \cap f(\sigma_r) \ne \emptyset$. 
So far we know that for $r$ a prime power $N_r(d) = (r-1)(d+1)$, and for $r$ not a prime power and $d$ sufficiently large $N_r(d) > (r-1)(d+1)$. 
In the following 
we establish some upper and lower bounds on the function~$N_r$.

\begin{theorem}
 Let $r \ge 2$ and $d \ge 1$ be integers. 
 Let $q \ge r$ be a prime power, and let $N=(q-1)(d+1)$\break$-(q-r) = (q-1)d+r-1$. 
 Then for any continuous map $f\colon \Delta_N \longrightarrow \R^d$ there are $r$ pairwise disjoint faces $\sigma_1, \dots, \sigma_r$ with $f(\sigma_1) \cap \dots \cap f(\sigma_r) \ne \emptyset$.
\end{theorem}
\begin{proof}
	Let $M = (q-1)(d+1)$ and think of~$\Delta_N$ as a subcomplex of~$\Delta_M$. 
	Extend $f$ to a continuous map $F\colon \Delta_M \longrightarrow \R^d$. 
	By the Topological Tverberg Theorem there are $q$ pairwise disjoint faces $\sigma_1, \dots, \sigma_q$ of~$\Delta_M$ with $F(\sigma_1) \cap \dots \cap F(\sigma_q) \ne \emptyset$. 
	Only $q-r$ vertices of~$\Delta_M$ are not contained in~$\Delta_N$, so at least $r$ of the faces $\sigma_1, \dots, \sigma_q$ are contained in~$\Delta_N$.
\end{proof}

\noindent
Let $r \ge 6$ be an integer. 
By Bertrand's postulate there is a prime strictly between $r-1$ and~$2r-4$, so  
\[N_r(d) \le (2r-6)(d+1).\] 
There are of course much more precise estimates available: For example, for any
$\varepsilon>0$ and sufficiently large $r$ (depending on $\varepsilon$) 
there is a prime between $r$ and $(1+\varepsilon)r$, so
\[N_r(d) \le (1+\varepsilon)r(d+1).\]
We refer to Lou \& Yao \cite{LouYao} for even stronger bounds.

Next we investigate the asymptotics of the function $\tfrac{N_r(d)}{d}$ for $d \to \infty$.

\begin{lemma}
\label{lem}
Let $d \ge 1$, $r \ge 2$ and $k \ge 2$ be integers, and let $N \ge (r-1)(d+1)$. 
Suppose that for every continuous map $F\colon \Delta_{k(N+1)-1} \longrightarrow \R^{k(d+1)-1}$ there are $r$ pairwise disjoint faces $\sigma_1, \dots, \sigma_r$ of $\Delta_{k(N+1)-1}$ with $F(\sigma_1) \cap \dots \cap F(\sigma_r) \ne \emptyset$.
Then for any continuous map $f\colon \Delta_N \longrightarrow \R^d$ there are $r$ pairwise disjoint faces $\tau_1, \dots, \tau_r$ of $\Delta_{N}$ such that $f(\tau_1) \cap \dots \cap f(\tau_r) \ne \emptyset$.
\end{lemma}

\begin{proof}
	Let $f\colon\Delta_N\longrightarrow\R^d$ be a continuous map.
	The simplex $\Delta_{k(N+1)-1}$ is isomorphic to the $k$-fold join~$(\Delta_N)^{*k}$. 
	Define $F\colon (\Delta_N)^{*k} \longrightarrow \R^{k(d+1)-1}$ by $$F(\lambda_1x_1 + \dots + \lambda_kx_k) = (\lambda_1, \dots, \lambda_{k-1}, \lambda_1f(x_1), \dots, \lambda_kf(x_k)).$$ 
	Then there are $r$ pairwise disjoint faces $\sigma_1, \dots, \sigma_r$ of~$\Delta_{k(N+1)-1}$ with 
	$F(\sigma_1) \cap \dots \cap F(\sigma_r) \ne \emptyset$. 
	Let $\lambda_1^{(i)}x_1^{(i)}+ \dots + \lambda_r^{(i)}x_r^{(i)} \in \sigma_i$ with $F(\lambda_1^{(1)}x_1^{(1)}+ \dots + \lambda_r^{(1)}x_r^{(1)}) = \dots = F(\lambda_1^{(r)}x_1^{(r)}+ \dots + \lambda_r^{(r)}x_r^{(r)})$. 
	Thus, $\lambda_j^{(1)} = \dots = \lambda_j^{(r)}$ for all $j=1, \dots, k$, where we have equality for the $\lambda_k^{(i)}$ since $\lambda_k^{(i)} = 1- \sum_{j=1}^{k-1} \lambda_j^{(i)}$. 
	
	At least one $\lambda_j^{(1)}$ is nonzero because $\sum_j \lambda_j^{(1)} = 1$. 
	Then since $\lambda_j^{(1)} = \dots = \lambda_j^{(r)}$ we have $f(x_j^{(1)}) = \dots = f(x_j^{(r)})$ and 
	the points $x_j^{(1)},\dots,x_j^{(r)}$ come from pairwise disjoint faces in~$\Delta_N$. 
\end{proof}

Define $\beta_r(d) = N_r(d) - (r-1)(d+1) \ge 0$. The function~$\beta_r$ measures to 
which extent the Topological Tverberg Conjecture fails in dimension~$d$ for a 
fixed~$r$. 
We show that if the Topological Tverberg Conjecture fails, then it fails at least linearly in~$d$.

\begin{corollary}
\label{cor:red}
 For any $r \ge 2$ and $d, k \ge 1$ we have that $N_r(k(d+1)-1) \ge kN_r(d)$. 
 Moreover, $\beta_r(k(d+1)-1) \ge k\beta_r(d)$.  
 This can be written as $\beta_r(d) \ge \frac{\beta_0}{d_0+1}(d+1)$ for $d = k(d_0+1)-1$ and $\beta_0 = \beta_r(d_0)$ for some fixed~$d_0 \ge 1$.
\end{corollary}
\begin{proof}
 Let $N = N_r(d)-1$, and let $f\colon \Delta_N \longrightarrow \R^d$ be a continuous such that for every collection of $r$ pairwise disjoint faces $\tau_1, \dots, \tau_r$ of $\Delta_{N}$ the intersection $f(\tau_1) \cap \dots \cap f(\tau_r) = \emptyset$ vanishes.

 Construct the function $F\colon (\Delta_N)^{\ast k} \longrightarrow \R^{k(d+1)-1}$ as in the proof of Lemma~\ref{lem}. 
 Then for every collection of $r$ pairwise disjoint faces $\sigma_1, \dots, \sigma_r$ in $(\Delta_N)^{\ast k}$ the intersection $F(\sigma_1) \cap \dots \cap F(\sigma_r) = \emptyset$ also vanishes.
 Therefore, $N_r(k(d+1)-1)$ is at least one larger than the dimension of $(\Delta_N)^{\ast k}$, that is, $N_r(k(d+1)-1) \ge k(N+1) = kN_r(d)$. 
 
 Moreover, 
 \[
 N_r(k(d+1)-1) \ge kN_r(d) = (r-1)(kd+k)+k\beta_r(d)
 \]
 and consequently $\beta_r(k(d+1)-1) \ge k\beta_r(d)$. 
 Now, $d = k(d_0+1)-1$ implies $k = \frac{d+1}{d_0+1}$. 
 Thus, $\beta_r(d) = \beta_r(k(d_0+1)-1) \ge k\beta_r(d_0) = \frac{\beta_0}{d_0+1}{d+1}$.
\end{proof}

\begin{theorem}
 Let $r \ge 2$. Suppose there is a $d_0 \ge 1$ for which $N_r(d_0) \ge \alpha_r \cdot (d_0+1)+c$ 
 for some $\alpha_r \ge r-1$ and~$c \ge 0$. Then $N_r(d) \ge (\alpha_r+\frac{c}{d_0+1})\cdot (d+1)$ 
 for $d = k(d_0+1)-1$ with~$k \in \{1,2,\dots\}$. 
\end{theorem}

\begin{proof}
 We have that $\beta_0 \ge (\alpha_r-(r-1))(d_0+1)+c$. Thus, by Corollary~\ref{cor:red} 
 \[
 \beta_r(d) \ge (\alpha_r-(r-1)+\frac{c}{d_0+1})\cdot (d+1)
 \] 
 for $d = k(d_0+1)-1$. 
 This implies that $N_r(d) \ge (\alpha_r+\frac{c}{d_0+1})\cdot (d+1)$.
\end{proof}

Let us conclude with the strongest plausible statement that remains.
Its refutation (or proof) would require methods that are different 
from the ones employed in~\cite{Mabillard2015} and \cite{MabillardWagner-III} 
or in the manuscript at hand.

\begin{conjecture}
	Let $r\ge 2$ and $d\ge 1$. Then
	\[
	N_r(d)\ = \
	\begin{cases}
		(r-1)(d+1), & \text{if }r\text{ is a power of a prime or }d \le r,\\
		r(d+1)-1,   & \text{otherwise}.
	\end{cases} 
	\]
\end{conjecture}  

\small
\providecommand{\noopsort}[1]{}
\providecommand{\href}[2]{#2}

\end{document}